\documentclass{amsart}
\usepackage{amssymb,MnSymbol}
\usepackage{amsthm,amsmath}
\usepackage{epic,cite,xcolor}

\title[Quasitrivial $n$-ary semigroups]{Every quasitrivial $n$-ary semigroup is reducible to a semigroup}

\author[M. Couceiro]{Miguel Couceiro}
\address{Universit\'e de Lorraine, CNRS, Inria, LORIA, F-54000 Nancy, France}
\email{miguel.couceiro[at]\{loria,inria\}.fr }

\author[J. Devillet]{Jimmy Devillet}
\address{Mathematics Research Unit, University of Luxembourg, Maison du Nombre, 6, avenue de la Fonte, L-4364 Esch-sur-Alzette, Luxembourg}
\email{jimmy.devillet[at]uni.lu}

\date{\today}

\theoremstyle{plain}
\newtheorem{theorem}{Theorem}[section]
\newtheorem{lemma}[theorem]{Lemma}
\newtheorem{proposition}[theorem]{Proposition}
\newtheorem{corollary}[theorem]{Corollary}
\newtheorem{fact}[theorem]{Fact}

\theoremstyle{definition}
\newtheorem{definition}[theorem]{Definition}

\newtheorem{remark}[theorem]{Remark}

\newcommand{\bfx}{\mathbf{x}}
\newcommand{\bfy}{\mathbf{y}}

\newcommand{\Cdot}{\boldsymbol{\cdot}}

\numberwithin{equation}{section}

\begin{document}

\thanks{The second author is supported by the Luxembourg National Research Fund under the project PRIDE 15/10949314/GSM}

\begin{abstract}
We show that every quasitrivial $n$-ary semigroup is reducible to a binary semigroup, and we provide necessary and sufficient conditions for such a reduction to be unique. These results are then refined in the case of symmetric $n$-ary semigroups. We also  explicitly determine the
sizes of these classes when the semigroups are defined on finite sets. As a byproduct of these enumerations, we obtain several new integer sequences.
\end{abstract}
\keywords{quasitrivial polyadic semigroup, reducibility, symmetry, enumeration.}

\subjclass{05A15, 20N15; 16B99, 20M14.}

\maketitle

\section{Introduction}

Let $X$ be a nonempty set and let $n \geq 2$ be an integer. In this paper we are interested in $n$-ary operations $F\colon X^n \to X$ that are \emph{associative}, {{namely}} that satisfy the following system of identities
\begin{multline*}
F(x_1,\ldots,x_{i-1},F(x_i,\ldots,x_{i+n-1}),x_{i+n},\ldots,x_{2n-1}) \\
=~ F(x_1,\ldots,x_i,F(x_{i+1},\ldots,x_{i+n}),x_{i+n+1},\ldots,x_{2n-1})
\end{multline*}
for all $x_1,\ldots,x_{2n-1}\in X$ and all $1\leq i\leq n-1$. This generalisation of associativity was originally proposed by  D\"{o}rnte \cite{Dor28} and studied by Post \cite{Pos40} in the framework of $n$-ary groups and their reductions. An operation $F\colon X^n \to X$ is said to be \emph{reducible to a binary operation} (resp.\ \emph{ternary operation}) if it can be written as a composition of a binary (resp.\ ternary) associative operation (see Definition \ref{de:FH5}).

Recently, the study of reducibility criteria for $n$-ary semigroups {{, that is, a set $X$ endowed with an associative $n$-ary operation,}} gained an increasing interest (see, e.g., \cite{DudMuk06,Ack,DevKiMar17,KiSom18}). In particular, Dudek and Mukhin \cite{DudMuk06} provided necessary and sufficient conditions under which an $n$-ary associative operation is reducible to a binary associative operation. Indeed, they proved (see \cite[Theorem 1]{DudMuk06}) that an associative operation $F\colon X^n \to X$ is reducible to an associative binary operation if and only if one can \emph{adjoin a neutral element $e$ to $X$ for $F$}, that is, there is an $n$-ary associative extension $\tilde{F}\colon \left(X\bigcup\{e\}\right)^n \to X\bigcup\{e\}$ of $F$ such that $e$ is a neutral element for $\tilde{F}$ and $\tilde{F}|_{X^n} = F$. In this case, a binary reduction $G^{e}$ of $F$  can be defined by
\begin{equation*}
G^{e}(x,y)=\tilde{F}(x,\underbrace{e,\ldots,e}_{n-2 ~\text{times}},y) \quad x,y \in X.
\end{equation*}

Recently, Ackerman \cite{Ack} also investigated reducibility criteria for $n$-ary associative operations that are \emph{quasitrivial}, that is, operations that preserve all unary relations: for every $x_1,\ldots,x_n\in X$, {{ $F(x_1,\ldots,x_n)\in\{x_1,\ldots,x_n\}$.}}

{{\begin{remark}
Quasitrivial operations are also called conservative operations \cite{PouRosSto96}. This property has been extensively used in the classification of constraint satisfaction problems into complexity classes (see, e.g, \cite{Bul} and the references therein).
\end{remark}}}

The following result reassembles Corollaries 3.14 and 3.15, and Theorem 3.18 of \cite{Ack}.
\begin{theorem}\label{thm:ack1}
Let $F\colon X^n \to X$ be an associative and quasitrivial operation.
\begin{enumerate}
\item[(a)] If $n$ is even, then $F$ is reducible to an associative and quasitrivial binary operation $G\colon X^2 \to X$.
\item[(b)] If $n$ is odd, then $F$ is reducible to an associative and quasitrivial ternary operation $H\colon X^3 \to X$.
\item[(c)] {{If $n=3$ and $F$ is not reducible to an associative binary operation $G\colon X^2 \to X$, then there exist $a_1,a_2\in X$ with $a_1\neq a_2$ such that
\begin{itemize}
\item $F|_{(X\setminus \{a_1,a_2\})^3}$ is reducible to an associative binary operation.
\item $a_1$ and $a_2$ are neutral elements for $F$.
\end{itemize}}}
\end{enumerate}
\end{theorem}

From Theorem \ref{thm:ack1} (c) it would follow that if an associative and quasitrivial operation $F\colon X^n \to X$ is not reducible to an associative binary operation $G\colon X^2 \to X$, then $n$ is odd and there exist distinct $a_1,a_2\in X$ that are neutral elements for $F$.

However, Theorem \ref{thm:ack1} (c) supposes the existence of a ternary associative and quasitrivial operation $H\colon X^3 \to X$ that is not reducible to an associative binary operation, and Ackerman did not provide any example of such an operation.

In this paper we show that there is no associative and quasitrivial $n$-ary operation that is not reducible to an associative binary operation (Corollary \ref{cor:main1}). Hence, for any associative and quasitrivial operation $F\colon X^n \to X$ one can adjoin a neutral element to $X$.
Now this raises the question of whether such a binary reduction is unique and whether it is quasitrivial. We show that both of these properties are  equivalent to the existence of at most one neutral element for the $n$-ary associative and quasitrivial operation (Theorem \ref{thm:unired}). Since an $n$-ary associative and quasitrivial operation has at most one neutral element when $n$ is even or at most two when $n$ is odd (Proposition \ref{prop:arity}), {{in the case when $X$ is finite}}, we also provide several enumeration results
(Propositions \ref{prop:CDM} and \ref{prop:qn}) that explicitly determine the sizes of the corresponding classes of associative and quasitrival $n$-ary operations in terms of the size of the underlying set $X$.  As a by-product, these enumeration results led to several integer sequences that were previously unknown in the Sloane's On-Line Encyclopedia of Integer Sequences (OEIS, see \cite{Slo}). These results are further refined in the case of symmetric operations (Theorem \ref{thm:sym}).

\section{Motivating results}

In this section we recall some basic definitions and present some motivating results. In particular, we show that every associative and quasitrivial operation $F\colon X^n \to X$ is reducible to an associative binary operation (Corollary \ref{cor:main1}).

Throughout this paper let $k \geq 1$ and $x\in X$. We use the shorthand notation $[k] = \{1,\ldots,k\}$ and  $n \Cdot x = x,\ldots,x$ ($n$ times), and we denote the set of all constant $n$-tuples over $X$ by $\Delta^n_X=\{(n\Cdot y)\mid \,y\in X\}.$ {{Also, we denote the size of any set $S$ by $|S|$}}.

Recall that a \emph{neutral element} for $F\colon X^n\to X$ is an element $e_F\in X$ such that
\[
F((i-1)\Cdot e_F,x,(n-i)\Cdot e_F) ~=~ x
\]
for all $x\in X$ and all $i\in [n]$. When the meaning is clear from the context, we may drop the index $F$ and denote a neutral element for $F$ by $e$.

\begin{definition}[see {\cite{Ack,DudMuk06}}]\label{de:FH5}
Let $G\colon X^2\to X$, and $H\colon X^3 \to X$ be associative operations.
\begin{enumerate}
 \item An operation $F\colon X^n\to X$ is said to be \emph{reducible to} $G$ if $F(x_1,\ldots,x_n)=G_{n-1}(x_1,\ldots,x_n)$ for all $x_1,\ldots,x_n\in X$, where $G_1 = G$ and
 \[G_{m}(x_1,\ldots,x_{m+1}) = G_{m-1}(x_1,\ldots,x_{m-1},G(x_m,x_{m+1})),\]
for each integer $2\leq m\leq n-1$. In this case, $G$ is said to be a \emph{binary reduction of $F$}.
\item Similarly, $F$ is said to be \emph{reducible to} $H$ if  $n$ is odd and $F(x_1,\ldots,x_n)=H_{n-3}(x_1,\ldots,x_n)$ for all $x_1,\ldots,x_n\in X$, where $H_0 = H$ and
\[H_{m}(x_1,\ldots,x_{m+3})  =  H_{m-2}(x_1,\ldots,x_m, H(x_{m+1},x_{m+2},x_{m+3})),\]
 for each even integer $2\leq m\leq{n-3}$.
In this case, $H$ is said to be a \emph{ternary reduction of $F$}.
\end{enumerate}
\end{definition}

As we will see, every associative and quasitrivial operation $F\colon X^n \to X$ is reducible to an associative binary operation. To show this, we will make use of the follwing auxiliary result.

\begin{lemma}[see {\cite[Lemma 1]{DudMuk06}}]\label{lemma:Dud3}
If $F\colon X^n\to X$ is associative and has a neutral element $e\in X$, then $F$ is reducible to the associative operation $G^{e}\colon X^2\to X$ defined by
\begin{equation}\label{eq:dud}
G^{e}(x,y)=F(x,(n-2)\Cdot e,y), \text{ for every } x,y \in X.
\end{equation}
\end{lemma}

{{The following corollary follows from Theorem \ref{thm:ack1} and Lemma \ref{lemma:Dud3}.}}

\begin{corollary}\label{cor:main1}
Every associative and quasitrivial operation $F\colon X^n \to X$ is reducible to an associative binary operation.
\end{corollary}

{{Theorem \ref{thm:ack1}(c) states that a ternary associative and quasitrivial operation $F\colon X^3\to X$ must have two neutral elements, whenever it is not reducible to a binary operation.}} In particular, we can show that two distinct elements $a_1,a_2\in X$ are neutral elements for $F$ if and only if they are neutral elements for the restriction $F|_{\{a_1,a_2\}^3}$ of $F$ to $\{a_1,a_2\}^3$. Indeed, the condition is obviously necessary, while its sufficiency  follows from the Lemma \ref{prop:aqred} below.

\begin{lemma}\label{prop:aqred}
Let $H\colon X^3 \to X$ be an associative and quasitrivial operation.
\begin{enumerate}
\item[(a)] If $a_1,a_2\in X$  are two distinct neutral elements for $H|_{\{a_1,a_2\}^3}$, then
\[
H(a_1,a_1,x) = H(x,a_1,a_1) = x = H(x,a_2,a_2) = H(a_2,a_2,x), \quad x\in X.
\]
\item[(b)] If $a_1,a_2\in X$ are {{two}} distinct neutral elements for $H|_{\{a_1,a_2\}^3}$, then both $a_1$ and $a_2$ are neutral elements for $H$.
\end{enumerate}
\end{lemma}

\begin{proof}
\begin{enumerate}
\item[(a)] Let $x\in X$. We only show that $H(a_1,a_1,x) = x$, since the other equalities can be shown similarly. Clearly, the equality holds when $x\in \{a_1,a_2\}$. So let $x\in X\setminus \{a_1,a_2\}$ and, for a contradiction, suppose that $H(a_1,a_1,x) = a_1$. By the associativity and quasitriviality of $H$, we then have
\begin{align*}
a_1 ~&=~ H(a_1,a_1,x) ~ = ~ H(a_1,H(a_1,a_2,a_2),x) \\
    ~&=~ H(H(a_1,a_1,a_2),a_2,x) ~ = ~ H(a_2,a_2,x) \in \{a_2,x\},
\end{align*}
which contradicts the fact that $a_1, a_2$ and $x$ are pairwise distinct.

\item[(b)] Suppose to the contrary that $a_1$ is not a neutral element for $H$ (the other case can be dealt with similarly). By Lemma \ref{prop:aqred}(a) we have that $H(a_1,a_1,y) = H(y,a_1,a_1) = y$ for all $y\in X$. By assumption, there exists $x\in X\setminus \{a_1,a_2\}$ such that $H(a_1,x,a_1) = a_1$. We have two cases to consider.
\begin{itemize}
\item If $H(a_2,x,a_2) = x$, then by Lemma \ref{prop:aqred}(a) we have that
\begin{align*}
H(x,a_2,a_1) ~&=~ H(H(x,a_1,a_1),a_2,a_1) ~ = ~ H(x,a_1,H(a_1,a_2,a_1)) \\
             ~&=~ H(x,a_1,a_2) ~ = ~ H(H(a_1,a_1,x),a_1,a_2) \\
             ~&=~ H(a_1,H(a_1,x,a_1),a_2) ~ = ~ H(a_1,a_1,a_2) ~ = ~ a_2.
\end{align*}
Also, by Lemma \ref{prop:aqred}(a) we have that
\begin{align*}
x ~&=~ H(x,a_1,a_1) ~=~ H(H(a_2,x,a_2),a_1,a_1) \\
  ~&=~ H(a_2,H(x,a_2,a_1),a_1) ~=~ H(a_2,a_2,a_1) ~=~ a_1,
\end{align*}
which contradicts the fact that $x\neq a_1$.
\item If $H(a_2,x,a_2) = a_2$, then by Lemma \ref{prop:aqred}(a) we have that
\begin{align*}
H(x,x,a_2) ~&=~ H(x,H(a_2,a_2,x),a_2) \\
           ~&=~ H(x,a_2,H(a_2,x,a_2)) ~ = ~ H(x,a_2,a_2) ~ = ~ x,
\end{align*}
and
\begin{align*}
H(a_1,x,x) ~&=~ H(a_1,H(x,a_1,a_1),x) \\
           ~&=~ H(H(a_1,x,a_1),a_1,x) ~ = ~ H(a_1,a_1,x) ~ = ~ x.
\end{align*}
By Lemma \ref{prop:aqred}(a) we also have that
\begin{align*}
x ~&=~ H(x,a_2,a_2) ~=~ H(H(a_1,x,x),a_2,a_2) \\
  ~&=~ H(a_1,H(x,x,a_2),a_2) ~=~ H(a_1,x,a_2) \\
  ~&=~ H(a_1,H(x,a_1,a_1),a_2) ~=~ H(H(a_1,x,a_1),a_1,a_2) \\
  ~&=~ H(a_1,a_1,a_2) ~=~ a_2,
\end{align*}
which contradicts the fact that $x\neq a_2$.\qedhere
\end{itemize}
\end{enumerate}
\end{proof}

We now present some geometric considerations of quasitrivial operations.
The {\it preimage} of an element $x\in X$ under an operation $F\colon X^n\to X$ is denoted by $F^{-1}[x]$. When $X$ is finite, {{namely}} $X=[k]$, we also define the \emph{preimage sequence of $F$} as the nondecreasing $k$-element sequence of the numbers $|F^{-1}[x]|$, $x\in [k]$. We denote this sequence by $|F^{-1}|$.

Recall that the \emph{kernel} of an operation $F\colon [k]^n\to [k]$ is the equivalence relation
$
\mathrm{Ker}(F) = \{\{\bfx,\bfy\}\mid F(\bfx)=F(\bfy)\}.
$
The \emph{contour plot} of $F\colon [k]^n\to [k]$ is the undirected graph
$\mathcal{C}_F = ([k]^n,E)$, where $E$ is the non-reflexive part of
$\mathrm{Ker}(F)$, {{that is}},
$
E = \{\{\bfx,\bfy\}\mid\bfx\neq\bfy~\text{and}~F(\bfx)=F(\bfy)\}.
$
We say that two tuples $\bfx,\bfy\in [k]^n$ are \emph{$F$-connected} (or simply \emph{connected}) if $\{\bfx,\bfy\}\in \mathrm{Ker}(F)$.

{{By an idempotent operation, we mean an operation $F\colon X^n \to X$ that satisfies $F(n\Cdot x)=x$ for all $x\in X$.}}

\begin{lemma}\label{lem:graph}
An operation $F \colon [k]^n \to [k]$ is quasitrivial if and only if it is idempotent and each $(x_1,...,x_n) \in [k]^n \setminus \Delta^n_{[k]}$ is connected to some $(n\Cdot x)\in \Delta^n_{[k]}$.
\end{lemma}

\begin{proof}
Clearly, $F$ is quasitrivial if and only if it is idempotent and for any $(x_1,...,x_n) \in [k]^n\setminus \Delta^n_{[k]}$ there exists $i \in \{1,...,n\}$ such that $F(x_1,...,x_n) = x_i = F(n \Cdot x_i)$.
\end{proof}

In the sequel we shall make use of the following two lemmas.

\begin{lemma}\label{lem:easy}
For each $x\in [k]$, the number of tuples $(x_1,\ldots,x_n)\in [k]^n$ with at least one component equal to $x$ is given by $k^n-(k-1)^n$.
\end{lemma}

\begin{proof}
Let $x\in [k]$. The set of tuples in $[k]^n$ with at least one component equal to $x$ is the set $[k]^n \setminus ([k] \setminus \{x\})^n$, and its cardinality is $k^n- (k - 1)^n$ since $([k] \setminus \{x\})^n \subseteq [k]^n$.
\end{proof}

\begin{lemma}\label{lem:deg}
Let $F\colon [k]^n\to [k]$ be a quasitrivial operation. Then, for each $x \in [k]$, we have $|F^{-1}[x]| \leq k^n-(k-1)^n$.
\end{lemma}

\begin{proof}
Let $x \in [k]$. Since $F\colon [k]^n\to [k]$ is quasitrivial, it follows from Lemma \ref{lem:graph} that the point $(n\Cdot x)$ is  at most connected to all $(x_1,...,x_n) \in [k]^n$ with at least one component equal to $x$. By Lemma \ref{lem:easy}, we conclude that there are exactly $k^n-(k-1)^n$ such points.
\end{proof}

Recall that an element $z\in X$ is said to be an \emph{annihilator} for $F$ if {{$F(x_1,...,x_n) = z$,}}
whenever $(x_1,...,x_n)\in X^n$ has at least one component equal to $z$.

\begin{remark}
A neutral element need not be unique when $n\geq 3$ (for instance, $F(x_1,x_2,x_3)\equiv x_1+x_2+x_3~(\mathrm{mod}~2)$ on $X=\mathbb{Z}_2$). However, if an annihilator exists, then it is unique.
\end{remark}

\begin{proposition}\label{prop:an}
Let $F\colon [k]^n\to [k]$ be a quasitrivial operation and let $z \in [k]$. Then $z$ is an annihilator if and only if $|F^{-1}[z]| = k^n-(k-1)^n$.
\end{proposition}
\begin{proof}
(Necessity) If $z$ is an annihilator, then we know that $F(i\Cdot z,x_{i+1},...,x_n) = z$ for all $i \in [n]$, all $x_{i+1},...,x_n \in [k]$ and all permutations of $(i\Cdot z,x_{i+1},...,x_n)$. Thus, by Lemma \ref{lem:easy} $(n\Cdot z)$ is connected to $k^n-(k-1)^n$ points. Finally, using Lemma \ref{lem:deg} we get $|F^{-1}[z]| = k^n-(k-1)^n$.

(Sufficiency) If $|F^{-1}[z]| = k^n-(k-1)^n$, then by Lemmas \ref{lem:graph} and \ref{lem:easy} we have that $(n\Cdot z)$ is connected to the $k^n-(k-1)^n$ points $(x_1,...,x_n) \in [k]^n$ containing at least one component equal to $z$. Thus, we have $F(i\Cdot z,x_{i+1},...,x_n) = z$ for all $i \in [n]$, all $x_{i+1},...,x_n \in [k]$ and all permutations of $(i\Cdot z,x_{i+1},...,x_n)$, which shows that $z$ is an annihilator.
\end{proof}

\begin{remark}
By Proposition \ref{prop:an}, if $F\colon [k]^n\to [k]$ is quasitrivial, then each element $x$ such that $|F^{-1}[x]| = k^n-(k-1)^n$ is unique.
\end{remark}

\section{Criteria for unique reductions and some enumeration results}

In this section we show that an associative and quasitrivial operation $F\colon X^n \to X$ is uniquely reducible to an associative and quasitrivial binary operation if and only if $F$ has at most one neutral element (Theorem \ref{thm:unired}). We also enumerate the class of associative and quasitrivial $n$-ary operations, which leads to a previously unknown sequence in the OEIS {{\cite{Slo}}} (Proposition \ref{prop:qn}).
Let us  first recall a useful result from \cite{DevKiMar17}.

\begin{lemma}[see {\cite[Proposition 3.5]{DevKiMar17}}]\label{lem:uni}
Assume that the operation $F\colon X^n\to X$ is associative and reducible to associative binary operations $G\colon X^2\to X$ and $G'\colon X^2\to X$. If $G$ and $G'$ are idempotent or have the same neutral element, then $G=G'$.
\end{lemma}

From Lemma~\ref{lem:uni}, we immediately get a necessary and sufficient condition that guarantees unique reductions for associative operation that have a neutral element.

\begin{corollary}\label{cor:triv}
Let $F\colon X^n\to X$ be an associative operation that is reducible to associative binary operations $G\colon X^2\to X$ and $G'\colon X^2\to X$ that have neutral elements. Then, $G=G'$ if and only if $G$ and $G'$ have the same neutral element.
\end{corollary}

Using Lemma \ref{lemma:Dud3}, Corollary \ref{cor:triv}, and observing that
\begin{enumerate}
\item[(i)] a binary associative operation has at most one neutral element, \item[(ii)] the neutral element of a binary reduction $G\colon X^2\to X$ of an associative operation $F\colon X^n\to X$ is also a neutral element for $F$, and
\item[(iii)] if $e$ is a neutral element for an associative operation $F\colon X^n\to X$ and $G\colon X^2\to X$ is a reduction of $F$, then  $G_{n-2}((n-1)\Cdot e)$ {{(see Definition \ref{de:FH5})}} is the neutral element for $G$,
\end{enumerate}
we can generalise Corollary \ref{cor:triv} as follows.

\begin{proposition}\label{prop:unired2}
Let $F\colon X^n\to X$ be an associative operation, and let $E_F$ be the set of its neutral elements and $R_F$ of its binary reductions. {{If $E_F\neq\varnothing$, then for any $G\in R_F$, there exists $e\in E_F$ such that $G=G^e$. Moreover, the mapping $\sigma\colon E_F\to R_F$ defined by $\sigma(e)=G^e$ is a bijection.}} In particular, $e$ is the unique neutral element for $F$ if and only if
$G^{e}$ is the unique binary reduction of $F$.
\end{proposition}

{{As we will see in Proposition \ref{prop:arity}, the size of $E_F$, and thus
of $R_F$, is at most 2 whenever $F$ is quasitrivial.}}

Let {{$Q_1^2(X)$}} denote the class of associative and quasitrivial operations $G\colon X^2 \to X$ that have exactly one neutral element, and let {{$A_1^2(X)$}} denote the class of associative operations $G\colon X^2 \to X$ that have exactly one neutral element $e_G\in X$ and that satisfy the following conditions:
\begin{itemize}
\item $G(x,x)\in\{e_G,x\}$ for all $x\in X$,
\item $G(x,y)\in \{x,y\}$ for all $(x,y)\in X^2\setminus \Delta^2_X$,
\item If there exists $x\in X\setminus\{e_G\}$ such that $G(x,x)=e_G$, then $x$ is unique and we have $G(x,y)=G(y,x)=y$ for all $y\in X\setminus\{x,e_G\}$.
\end{itemize}
Note that $Q_1^2(X)=A_1^2(X)=X^{X^2}$ when $|X|=1$. Also, it is not difficult to see that $Q_1^2(X)\subseteq A_1^2(X)$. Actually, we have that $G\in Q_1^2(X)$ if and only if $G\in A_1^2(X)$ and $|{G}^{-1}[e]| = 1$, where $e$ is the neutral element for $G$. A characterization of the class of associative and quasitrivial binary operations as well as $Q_1^2(X)$ can be found in \cite[Theorem 2.1, Fact 2.4]{CouDevMar2}.

Recall that two groupoids $(X,G)$ and $(Y,G')$ are said to be \emph{isomorphic}, and we denote it by $(X,G) \simeq (Y,G')$, if there exists a bijection $\phi\colon X \to Y$ such that {{$\phi(G(x,y)) = G'(\phi(x),\phi(y))$ for every $x,y\in X$.}}
{{The following straightforward proposition states, in particular, that any $G\in A_1^2(X)\setminus Q_1^2(X)$ gives rise to a semigroup which has a unique $2$-element subsemigroup isomorphic to the additive semigroup on $\mathbb{Z}_2$.}}

\begin{proposition}\label{prop:obv}
Let $G\colon X^2 \to X$ be an operation. Then $G \in A_1^2(X)\setminus Q_1^2(X)$ if and only if there exists a unique pair $(x,y)\in X^2\setminus \Delta^2_X$ such that the following conditions hold
\begin{enumerate}
\item[(a)]
$(\{x,y\},G|_{\{x,y\}^2}) \simeq (\mathbb{Z}_2,+),$
\item[(b)] $G|_{(X\setminus\{x,y\})^2}$ is associative and quasitrivial, and
\item[(c)] every $z\in X\setminus\{x,y\}$ is an annihilator for $G|_{\{x,y,z\}^2}$.
\end{enumerate}
\end{proposition}

\begin{proposition}\label{prop:aqe}
Let $F\colon X^n\to X$ be an associative and quasitrivial operation. Suppose that $e \in X$ is a neutral element for $F$.
\begin{enumerate}
\item[(a)] If $n$ is even, then $F$ is reducible to an operation $G\in Q_1^2(X)$.
\item[(b)] If $n$ is odd, then $F$ is reducible to the operation $G^e\in A_1^2(X)$.
\end{enumerate}
\end{proposition}

\begin{proof}
(a) By Theorem \ref{thm:ack1}(a) we have that $F$ is reducible to an associative and quasitrivial binary operation $G\colon X^2 \to X$. Finally, we observe that $G_{n-2}((n-1)\Cdot e)$ is the neutral element for $G$.

(b) {{By Lemma \ref{lemma:Dud3} we have that $F$ is reducible to an associative operation $G^e\colon X^2\to X$ of the form \eqref{eq:dud} and that $e$ is also a neutral element for $G^e$. Since $F$ is quasitrivial, it follows from \eqref{eq:dud} that $G^e(x,x) \in \{x,e\}$ for all $x \in X$. If $|X|=2$, then the proof is complete. So suppose that $|X|>2$ and let us show that $G^e(x,y)\in \{x,y\}$ for all $(x,y)\in X^2 \setminus \Delta^2_X$.}} Since $e$ is a neutral element for $G^e$, we have that $G^e(x,e) = G^e(e,x) = x$ for all $x\in X\setminus \{e\}$. So suppose to the contrary that there are distinct $x,y \in X\setminus \{e\}$ such that  $G^e(x,y)\not \in \{x,y\}$. As $G^e$ is a reduction of $F$ and $F$ is quasitrivial, we must have $G^e(x,y)= e$. But then, using the associativity of $G^e$, we have that
\[
y = G^e(e,y) = G^e(G^e(x,y),y) = G^e(x,G^e(y,y)) \in \{G^e(x,y),G^e(x,e)\} = \{e,x\},
\]
which contradicts the fact that $x$, $y$ and $e$ are pairwise distinct.

Now, suppose that there exists $x\in X\setminus\{e\}$ such that $G^e(x,x)=e$ and let $y\in X\setminus \{x,e\}$. Since
\[
y = G^e(e,y) = G^e(G^e(x,x),y) = G^e(x,G^e(x,y)),
\]
we must have $G^e(x,y) = y$. Similarly, we can show that $G^e(y,x) = y$.

To complete the proof, we only need to show that such an $x$ is unique. Suppose to the contrary that there exists $x'\in X\setminus\{x,e\}$ such that $G^e(x',x')=e$. Since $x,x'$ and $e$ are pairwise distinct and
\[
x' = G^e(e,x') = G^e(G^e(x,x),x') = G^e(x,G^e(x,x')),
\]
and
\[
x = G^e(x,e) = G^e(x,G^e(x',x')) = G^e(G^e(x,x'),x'),
\]
we must have $x=G^e(x,x')=x'$, which yields the desired contradiction.
\end{proof}

We observe that the associative operation $F\colon \mathbb{Z}_2^n \to \mathbb{Z}_2$ defined by
\[
F(x_1,\ldots,x_n) \equiv \sum_{i=1}^{n}x_i ~(\mathrm{mod}~ 2), \qquad x_1,\ldots,x_n \in \mathbb{Z}_2,
\]
has 2 neutral elements, namely $0$ and $1$, when $n$ is odd. Moreover, it is quasitrivial if and only if $n$ is odd. This also illustrates the fact that an associative and quasitrivial $n$-ary operation that has 2 neutral elements does not necessarily have a quasitrivial reduction. Indeed, when $n$ is odd, $G(x_1,x_2)\equiv x_1+x_2~(\mathrm{mod}~2)$ and  $G'(x_1,x_2)\equiv x_1+x_2+1~(\mathrm{mod}~2)$ on $X=\mathbb{Z}_2$ are two distinct reductions of $F$ but neither is quasitrivial.

Clearly, if an associative operation $F\colon X^n\to X$ is reducible to an operation $G\in Q_1^2(X)$, then it is quasitrivial. The following proposition provides a necessary and sufficient condition for $F$ to be quasitrivial when $G\in A_1^2(X)\setminus Q_1^2(X)$.

\begin{proposition}\label{prop:blue}
Let $F\colon X^n\to X$ be an associative operation. Suppose that $F$ is reducible to an operation $G\in A_1^2(X)\setminus Q_1^2(X)$. Then, $F$ is quasitrivial if and only if $n$ is odd.
\end{proposition}

\begin{proof}
To show that the condition is necessary, let $x \in X\setminus\{e\}$ such that $G(x,x)=e$. If $n$ is even, then $F(n \Cdot x) = G_{\frac{n}{2}-1}(\frac{n}{2}\Cdot G(x,x)) = e$, contradicting quasitriviality.

So let us prove that the condition is also sufficient. Note that $G\in A_1^2(X)\setminus Q_1^2(X)$, and thus we only need to show that $F$ is idempotent. Since $F$ is reducible to $G$, we clearly have that $F(n\Cdot x) = x$ for all $x\in X$ such that $G(x,x) = x$.

Let $y \in X\setminus\{e\}$ such that $G(y,y)=e$. Since $n$ is odd, we have that
\[
F(n\Cdot y) = G\Bigl(\, y,G_{\frac{n-1}{2}-1}\Bigl(\, \frac{n-1}{2}\Cdot G(y,y)\, \Bigr)\, \Bigr) = G(y,e) = y.
\]
Hence, $F$ is idempotent and the proof is now complete.
\end{proof}

It is not difficult to see that the operation $F\colon \mathbb{Z}_{n-1}^n \to \mathbb{Z}_{n-1}$ defined by
\[
F(x_1,\ldots,x_n) \equiv \sum_{i=1}^{n}x_i ~(\mathrm{mod}~ (n-1)), \qquad x_1,\ldots,x_n \in \mathbb{Z}_{n-1},
\]
is associative, idempotent, symmetric {{(that is, $F(x_1,\ldots,x_n)$ is invariant under any permutation of $x_1,\ldots,x_n$),}} and has $n-1$ neutral elements. However, this number is much smaller for quasitrivial operations.

\begin{proposition}\label{prop:arity}
Let $F\colon X^n \to X$ be an associative and quasitrivial operation.
\begin{enumerate}
\item[(a)] If $n$ is even, then $F$ has at most one neutral element.
\item[(b)] If $n$ is odd, then $F$ has at most two neutral elements.
\end{enumerate}
\end{proposition}

\begin{proof}
(a) By Theorem \ref{thm:ack1}(a) we have that $F$ is reducible to an associative and quasitrivial binary operation $G\colon X^2 \to X$. Suppose that $e_1,e_2\in X$ are two neutral elements for $F$. Since $G$ is quasitrivial we have
\begin{align*}
e_2 ~&=~ F((n-1)\Cdot e_1,e_2) ~ = ~ G(G_{n-2}((n-1)\Cdot e_1),e_2) \\
    ~&=~ G(e_1,e_2) ~=~ G(e_1,G_{n-2}((n-1)\Cdot e_2)) ~=~ F(e_1,(n-1)\Cdot e_2) ~=~ e_1.
\end{align*}
Hence, $F$ has at most one neutral element.

(b) By Theorem \ref{thm:ack1}(b) we have that $F$ is reducible to an associative and quasitrivial ternary operation $H\colon X^3 \to X$. For a contradiction, suppose that $e_1,e_2,e_3\in X$ are three {distinct}
 neutral elements for $F$. Since $H$ is quasitrivial, it is not difficult to see that $e_1$, $e_2$, and $e_3$ are neutral elements for $H$. Also, by Proposition \ref{prop:aqe}(b) we have that $H$ is reducible to the operations $G^{e_1}, G^{e_2}, G^{e_3}\in A_1^2(X)$. In particular, we have
\begin{multline*}
G^{e_1}(e_2,e_3) = G^{e_1}(G^{e_1}(e_1,e_2),e_3) \\
= H(e_1,e_2,e_3) = G^{e_2}(G^{e_2}(e_1,e_2),e_3) = G^{e_2}(e_1,e_3)
\end{multline*}
and
\[
H(e_1,e_2,e_3) = G^{e_3}(e_1,G^{e_3}(e_2,e_3)) = G^{e_3}(e_1,e_2).
\]
Hence, $H(e_1,e_2,e_3) \in \{e_2,e_3\}\cap\{e_1,e_3\}\cap\{e_1,e_2\},$ which shows that $e_1,e_2,e_3$ are not pairwise distinct, and thus yielding the desired contradiction.
\end{proof}

{{\begin{corollary}\label{prop:countaqe}
Let $F\colon X^n\to X$ be an operation and let $e_1$ and $e_2$ be distinct elements of $X$. Then $F$ is associative, quasitrivial, and has exactly the two neutral elements $e_1$ and $e_2$ if and only if $n$ is odd and $F$ is reducible to exactly the two operations $G^{e_1},G^{e_2}\in A_1^2(X)\setminus Q_1^2(X)$.
\end{corollary}}}

\begin{proof}
(Necessity) This follows from Propositions \ref{prop:unired2}, \ref{prop:aqe}, and \ref{prop:arity} together with the observation that $G^{e_1}(e_2,e_2) = e_1$ and $G^{e_2}(e_1,e_1)=e_2$.

(Sufficiency) This follows from Propositions \ref{prop:unired2} and \ref{prop:blue}.
\end{proof}

We can now state and prove the main result of this section.

{{\begin{theorem}\label{thm:unired}
Let $F\colon X^n\to X$ be an associative and quasitrivial operation. The following assertions are equivalent.
\begin{enumerate}
\item[(i)] Any binary reduction of $F$ is idempotent.
\item[(ii)] Any binary reduction of $F$ is quasitrivial.
\item[(iii)] $F$ has at most one binary reduction.
\item[(iv)] $F$ has at most one neutral element.
\item[(v)] $F((n-1)\Cdot x,y) ~ = ~ F(x,(n-1)\Cdot y)$ for any $x,y\in X$.
\end{enumerate}
\end{theorem}}}

\begin{proof} The implications $\textrm{(i)}\Rightarrow (\textrm{(ii)}~{{\text{and}}}~ \textrm{(v)})$ and $\textrm{(v)}\Rightarrow \textrm{(iv)}$ are straightforward.
By Proposition \ref{prop:arity} and Corollary \ref{prop:countaqe} we also have the implications $(\textrm{(ii)} ~{{\text{or}}}~ \textrm{(iii)}) \Rightarrow \textrm{(iv)}$. Hence, to complete the proof, it suffices to show that {$\textrm{(iv)} \Rightarrow (\textrm{(i)} ~{{\text{and}}}~ \textrm{(iii)})$}. First, we prove that $\textrm{(iv)} \Rightarrow \textrm{(i)}$.
We consider the two possible cases.

If $F$ has a unique neutral element $e$, then by Proposition \ref{prop:unired2} $G=G^e$ is the unique reduction of $F$ with neutral element $e$. For the sake of a contradiction, suppose that $G$ is not idempotent.
By Proposition \ref{prop:aqe} we then have that $n$ is odd and $G\in A^2_1(X)\setminus Q^2_1(X)$.

So let $x\in X\setminus \{e\}$ such that $G(x,x)\neq x$. Since $G=G^e$, we must have $G(x,x)=e$.
It is not difficult to see that $F(y,(n-1)\Cdot x) = y = F((n-1)\Cdot x,y)$ for all $y\in X$.
Now, if  there is $i\in \{2,\ldots,n-1\}$ such that
\[
F((i-1)\Cdot x,e,(n-i)\Cdot x) = x,
\]
then we have that $i-1$ and $n-i$ are both even or both odd (since $n$ is odd), and thus
\[
x = F((i-1)\Cdot x,e,(n-i)\Cdot x) \in \{G_2(x,e,x),G_2(e,e,e)\} = \{e\},
\]
which contradicts our assumption that $x\neq e$.
Hence,  we have $F((i-1)\Cdot x,e,(n-i)\Cdot x) = e$ for all $i\in \{1,\ldots,n\}$.

Now, if $|X|=2$, then the proof is complete since $e$ and $x$ are both neutral elements for $F$, which contradicts our assumption. So suppose that $|X|>2$.

Since $e$ is the unique neutral element for $F$, there exist $y\in X\setminus \{e,x\}$ and $i\in \{2,\ldots,n-1\}$ such that
\[
F((i-1)\Cdot x,y,(n-i)\Cdot x) = x.
\]
Again by the fact that $n$ is odd, $i-1$ and $n-i$ are both even or both odd, and thus
\[
x = F((i-1)\Cdot x,y,(n-i)\Cdot x) \in \{G_2(x,y,x),G_2(e,y,e)\} = \{G_2(x,y,x),y\}.
\]
Since $x\neq y$, we thus have that  $G_2(x,y,x) = x$. But then
\begin{align*}
e &= G(x,x) = G(x,G_2(x,y,x)) \\
  &= G(G(x,x),G(y,x)) = G(e,G(y,x)) = G(y,x) \in \{x,y\},
\end{align*}
which contradicts our assumption that $x,y,$ and $e$ are pairwise distinct.

Now, suppose that $F$ has no neutral element and that $G$ is a reduction of F that is not idempotent.
So let $x\in X$ such that $G(x,x)\neq x$, and let $y\in X\setminus\{x,G(x,x)\}$.
By the quasitriviality of $F$ we have $F((n-1)\Cdot x,y) \in \{x,y\}$. On the other hand, by the quasitriviality (and hence idempotency) of $F$ and the associativity of $G$ we have
\begin{align*}
F((n-1)\Cdot x,y) &= F(F(n\Cdot x),(n-2)\Cdot x,y) \\
                  &= G(G_{n-2}(G_{n-1}(n\Cdot x),(n-2)\Cdot x),y) \\
                  &= G(G_{2n-3}((2n-2)\Cdot x),y) \\
                  &= G(G_{n-2}((n-1)\Cdot G(x,x)),y) \\
                  &= F((n-1)\Cdot G(x,x),y) \in \{G(x,x),y\}.
\end{align*}
Since $x, G(x,x),$ and $y$ are pairwise distinct, it follows that $F((n-1)\Cdot x,y) = y$, which implies that $G(G_{n-2}((n-1)\Cdot x),y) = y$. Similarly, we can show that
\[
G(y,G_{n-2}((n-1)\Cdot x)) = y.
\]
Also, it is not difficult to see that
\[
G(G_{n-2}((n-1)\Cdot x),G(x,x)) = G(x,x) = G(G(x,x),G_{n-2}((n-1)\Cdot x)).
\]
Furthermore, since $F$ is idempotent and reducible to $G$, we also have that
\[
 G(G_{n-2}((n-1)\Cdot x),x) = x = G(x,G_{n-2}((n-1)\Cdot x)).
\]
Thus $G_{n-2}((n-1)\Cdot x)$ is a neutral element for $G$ and therefore a neutral element for $F$, which contradicts our assumption that $F$ has no neutral element.

As both cases yield a contradiction, we conclude that $G$ must be idempotent.  The implication $\textrm{(iv)} \Rightarrow \textrm{(iii)}$ is an immediate consequence of the implication $\textrm{(iv)} \Rightarrow \textrm{(i)}$ together with Lemma \ref{lem:uni}. Thus, the proof of Theorem \ref{thm:unired} is now complete.
\end{proof}

\begin{remark}
We observe that an alternative necessary and sufficient condition for the quasitriviality of a binary reduction of an $n$-ary quasitrivial semigroup has also been provided in \cite[Corollary 3.16]{Ack}.
\end{remark}

Theorem \ref{thm:unired} together with Corollary \ref{cor:main1} imply the following result.

\begin{corollary}\label{cor:need}
Let $F\colon X^n\to X$ be an operation. Then $F$ is associative, quasitrivial, and has at most one neutral element if and only if it is reducible to an associative and quasitrivial operation $G\colon X^2 \to X$. In this case,  $G$ is defined by $G(x,y)=F(x,(n-1)\Cdot y)$.
\end{corollary}

Recall that a \emph{weak ordering on $X$} is a binary relation $\lesssim$ on $X$ that is total and transitive (see, e.g.,
\cite{Kra} p. 14). We denote the symmetric part of $\lesssim$ by $\sim$. Also, a \emph{total ordering on $X$} is a weak ordering on $X$ that is antisymmetric.
If $(X,\lesssim)$ is a weakly ordered set, an element $a\in X$ is said to be \emph{maximal} \emph{for $\lesssim$} if $x\lesssim a$ for all $x\in X$. We denote the set of maximal elements of $X$ for $\lesssim$ by $\mathcal{M}_{\lesssim}(X)$.

Given a weak ordering $\lesssim$ on $X$, the $n$-ary \emph{maximum operation on $X$ for $\lesssim$} is the partial symmetric $n$-ary operation $\max^n_{\lesssim}$ defined on
\[
X^n\setminus\{(x_1,\ldots,x_n)\in X^n: \mbox{$|\mathcal{M}_{\lesssim}(\{x_1,\ldots,x_n\})| \geq 2$}\}
\]
by $\max^n_{\lesssim}(x_1,\ldots,x_n)=x_i$ where $i\in [n]$ is such that $x_j\lesssim x_i$ for all $j\in [n]$. If $\lesssim$ reduces to a total ordering, then clearly the operation $\max^n_{\lesssim}$ is defined everywhere on $X^n$. Also, the \emph{projection operations} $\pi_1\colon X^n\to X$ and $\pi_n\colon X^n\to X$ are respectively defined by $\pi_1(x_1,\ldots,x_n)=x_1$ and $\pi_n(x_1,\ldots,x_n)=x_n$ for all $x_1,\ldots,x_n\in X$.

Corollary \ref{cor:need} together with \cite[Theorem 1]{Lan80} and \cite[Corollary 2.3]{CouDevMar2} imply the following characterization of the class of quasitrivial $n$-ary semigroups with at most one neutral element.

\begin{theorem}\label{thm:char}
Let $F\colon X^n \to X$ be an operation. Then $F$ is associative, quasitrivial, and has at most one neutral element if and only if there exists a weak ordering $\lesssim$ on $X$ and a binary reduction $G\colon X^2 \to X$ of $F$ such that
\begin{equation}\label{eq:kimura}
G|_{A\times B} ~=~
\begin{cases}
\pi_1|_{A\times B}\hspace{1.5ex}\text{or}\hspace{1.5ex}\pi_2|_{A\times B}, & \text{if ~$A = B$},\\
\max^2_{\lesssim}|_{A\times B}, & \text{otherwise},
\end{cases}
\qquad \forall A,B\in {X/\sim}.
\end{equation}

Moreover, when $X=[k]$, then the weak ordering $\lesssim$ is uniquely defined as follows:
\begin{equation}\label{eq:deg}
x\lesssim y\quad\Leftrightarrow\quad |G^{-1}[x]|\,\leq\, |G^{-1}[y]|,\qquad x,y\in [k].
\end{equation}
\end{theorem}

Now, let us illustrate Theorem \ref{thm:char} for binary operations by means of their contour plots. We can always represent the contour plot of any operation $G\colon [k]^2\to [k]$ by fixing a total ordering on $[k]$.  In Figure~\ref{fig:2ab} (left),  we represent the contour plot of an operation $G\colon X^2\to X$ using the usual total ordering $\leq$ on $X=\{1,2,3,4\}$. {{To simplify the representation of the connected components, we omit edges that can be obtained by transitivity.}} It is not difficult to see that $G$ is quasitrivial. {{To check whether $G$ is associative, by Theorem \ref{thm:char}, it suffices to show that $G$ is of the form \eqref{eq:kimura} where the weak ordering $\lesssim$ is defined on $X$ by \eqref{eq:deg}}}. In Figure~\ref{fig:2ab} (right) we represent the contour plot of $G$ using the weak ordering $\lesssim$ on $X$ defined by \eqref{eq:deg}. We observe that $G$ is of the form \eqref{eq:kimura} for $\lesssim$ and thus by Theorem \ref{thm:char} it is associative.
\setlength{\unitlength}{3.5ex}
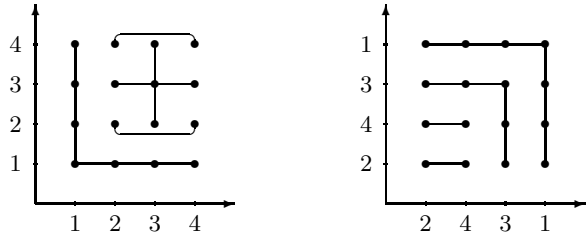
\begin{figure}[htbp]
\begin{center}
\begin{small}
\null\hspace{0.03\textwidth}
\begin{picture}(6,6)
\put(0.5,0.5){\vector(1,0){5}}\put(0.5,0.5){\vector(0,1){5}}
\multiput(1.5,0.45)(1,0){4}{\line(0,1){0.1}}%
\multiput(0.45,1.5)(0,1){4}{\line(1,0){0.1}}%
\put(1.5,0){\makebox(0,0){$1$}}\put(2.5,0){\makebox(0,0){$2$}}\put(3.5,0){\makebox(0,0){$3$}}
\put(4.5,0){\makebox(0,0){$4$}}
\put(0,1.5){\makebox(0,0){$1$}}\put(0,2.5){\makebox(0,0){$2$}}\put(0,3.5){\makebox(0,0){$3$}}
\put(0,4.5){\makebox(0,0){$4$}}
\multiput(1.5,1.5)(0,1){4}{\multiput(0,0)(1,0){4}{\circle*{0.2}}}
\drawline[1](1.5,4.5)(1.5,1.5)(4.5,1.5)\drawline[1](2.5,3.5)(4.5,3.5)\drawline[1](3.5,2.5)(3.5,4.5)
\put(3.5,4.5){\oval(2,0.5)[t]}\put(3.5,2.5){\oval(2,0.5)[b]}
\end{picture}
\hspace{0.1\textwidth}
\begin{picture}(6,6)
\put(0.5,0.5){\vector(1,0){5}}\put(0.5,0.5){\vector(0,1){5}}
\multiput(1.5,0.45)(1,0){4}{\line(0,1){0.1}}%
\multiput(0.45,1.5)(0,1){4}{\line(1,0){0.1}}%
\put(1.5,0){\makebox(0,0){$2$}}\put(2.5,0){\makebox(0,0){$4$}}\put(3.5,0){\makebox(0,0){$3$}}
\put(4.5,0){\makebox(0,0){$1$}}
\put(0,1.5){\makebox(0,0){$2$}}\put(0,2.5){\makebox(0,0){$4$}}\put(0,3.5){\makebox(0,0){$3$}}
\put(0,4.5){\makebox(0,0){$1$}}
\multiput(1.5,1.5)(0,1){4}{\multiput(0,0)(1,0){4}{\circle*{0.2}}}
\drawline[1](1.5,2.5)(2.5,2.5)\drawline[1](1.5,1.5)(2.5,1.5)\drawline[1](1.5,4.5)(4.5,4.5)(4.5,1.5)\drawline[1](1.5,3.5)(3.5,3.5)(3.5,1.5)
\end{picture}
\end{small}
\caption{An associative and quasitrivial binary operation $G$ on $X=\{1,2,3,4\}$
 whose values on $(1,1), (2,2), (3,3)$ and $(4,4)$ are $1,2,3$ and $4$, respectively.}
\label{fig:2ab}
\end{center}
\end{figure}

Let $\leq$ be a total ordering on $X$. An operation $F\colon X^n \to X$ is said to be \emph{$\leq$-preserving} if $F(x_1,\ldots,x_n)\leq F(x'_1,\ldots,x'_n)$, whenever $x_i\leq x'_i$ for all $i\in [n]$. Some associative binary operations $G\colon X^2 \to X$ are $\leq$-preserving for any total ordering on $X$ (e.g., $G(x,y) = x$ for all $x,y\in X$). However, there is no total ordering $\leq$ on $X$ for which an operation $G\in A_1^2(X)\setminus Q_1^2(X)$ is $\leq$-preserving. A typical example is the binary addition modulo 2.

\begin{proposition}
Suppose $|X|\geq 2$. If $G\in A_1^2(X)\setminus Q_1^2(X)$, then there is no total ordering $\leq$ on $X$ that is preserved by $G$.
\end{proposition}

\begin{proof}
Let $e\in X$ be the neutral element for $G$ and let $x\in X\setminus\{e\}$ such that $G(x,x)=e$. Suppose to the contrary that there exists a total ordering $\leq$ on $X$ such that $G$ is $\leq$-preserving. If  $x<e$, then  $e = G(x,x) \leq G(x,e) = x$, which contradicts our assumption. The case $x>e$ yields a similar contradiction.
\end{proof}

\begin{remark}\label{rem:ND} It is not difficult to see that any $\leq$-preserving operation $F\colon X^n \to X$ has at most one neutral element. Therefore, by Corollary \ref{cor:main1} and Theorem \ref{thm:unired} we conclude that any associative, quasitrivial, and $\leq$-preserving operation $F\colon X^n \to X$ is reducible to an associative, quasitrivial, and $\leq$-preserving operation $G\colon X^2 \to X$. For a characterization of  the class of associative, quasitrivial, and $\leq$-preserving operations $G\colon X^2 \to X$, see \cite[Theorem 4.5]{CouDevMar2}.
\end{remark}

We now provide several enumeration results that give the sizes of the classes of associative and quasitrivial operations that were considered above when $X=[k]$.
Recall that for any integers $0\leq \ell \leq k$, the \emph{Stirling number of the second kind} ${k\brace \ell}$ is defined by
\[
{k\brace \ell} ~=~ \frac{1}{\ell!}{\,}\sum_{i=0}^\ell(-1)^{\ell-i}{\ell\choose i}{\,}i^{k}.
\]
For any integer $k\geq 0$, let $q^2(k)$ (resp.\ $q^n(k)$) denote the number of associative and quasitrivial binary (resp.\ $n$-ary) operations on $[k]$. For any integer $k\geq 1$, we denote by $q_1^2(k)$ the cardinality of $Q_1^2([k])$. Also, we denote by $a_1^2(k)$ the cardinality of $A_1^2([k])$. By definition, we have $a_1^2(1)=1$. In \cite{CouDevMar2} the authors solved several enumeration problems concerning associative and quasitrivial binary operations. In particular, they computed $q^2(k)$ (see \cite[Theorem 4.1]{CouDevMar2}) as well as $q_1^2(k)$ (see \cite[Proposition 4.2]{CouDevMar2}). These sequences were also introduced in the OEIS \cite{Slo} as $A292932(k)$ and $A292933(k)$. The following result summarizes \cite[Theorem 4.1]{CouDevMar2} and \cite[Proposition 4.2]{CouDevMar2}.

\begin{proposition}\label{prop:CDM}
For any integer $k\geq 0$, we have the closed-form expression
\begin{equation*}\label{eq:Nk}
q^2(k) ~=~ \sum_{i=0}^k 2^i{\,}\sum_{\ell=0}^{k-i}(-1)^{\ell}{\,}{k\choose \ell}{k-\ell\brace i}{\,}(i+\ell)!{\,},\qquad k\geq 0,
\end{equation*}
where $q^2(0)=q^2(1)=1$. Moreover, for any integer $k\geq 1$, we have $q_1^2(k)=k{\,}q^2(k-1)$.
\end{proposition}

\begin{proposition}\label{prop:ae}
For any integer $k\geq 2$, we have $a_1^2(k) = kq^2(k-1) + k(k-1)q^2(k-2)$.
\end{proposition}

\begin{proof}
We already have that $Q_1^2([k])\subseteq A_1^2([k])$. Now, let us show how to construct an operation $G\in A_1^2([k])\setminus Q_1^2([k])$. There are $k$ ways to choose the element $x\in [k]$ such that $G(x,x) = e$ and $G(x,y) = G(y,x) = y$ for all $y\in [k]\setminus \{x,e\}$. Then we observe that the restriction of $G$ to $([k]\setminus\{x\})^2$ belongs to $Q_1^2([k]\setminus\{x\})$, so we have $q_1^2(k-1)$ possible choices to construct this restriction. This shows that $a_1^2(k)=q_1^2(k) + kq_1^2(k-1)$. Finally, by Proposition \ref{prop:CDM} we conclude that $a_1^2(k)=kq^2(k-1) + k(k-1)q^2(k-2)$.
\end{proof}

For any integer $k\geq 1$ let $q^n_1(k)$ (resp.\ $q^n_{0}(k)$) denote the number of associative and quasitrivial $n$-ary operations that have exactly one neutral element (resp.\ that have no neutral element) on $[k]$. Also, for any integer $k\geq 1$, let $q^n_{2}(k)$ denote the number of associative and quasitrivial $n$-ary operations that have two neutral elements on $[k]$. Clearly, $q^n(1)=q^n_1(1)=1$ and $q^n_{2}(1)=0$. The following proposition provides explicit forms of the latter sequences. Table \ref{tab:q} below provides the first few values of all the previously considered sequences. {{In view of Corollary \ref{prop:countaqe}, we only consider the case where $n$ is odd for $q^n_2(k)$ and $q^n(k)$.}}
\begin{proposition}\label{prop:qn}
For any integer $k\geq 1$ we have $q^n_1(k) = q_1^2(k)$ and $q^n_{0}(k) = q^2(k)-q_1^2(k)$. Also, for any integer $k\geq 2$ we have
\[
q^n_{2}(k) = \left\{
    \begin{matrix}
        0 & \mbox{if}~ n ~\mbox{is even} \\
        \binom{k}{2}q^2(k-2) & \mbox{if}~ n ~\mbox{is odd.}
    \end{matrix}
\right.
\]
and
\[
q^n(k) = \left\{
    \begin{matrix}
        q^2(k) & \mbox{if}~ n ~\mbox{is even} \\
        q^2(k) + \binom{k}{2}q^2(k-2) & \mbox{if}~ n ~\mbox{is odd.}
    \end{matrix}
\right.
\]
\end{proposition}

\begin{proof}
By Theorem \ref{thm:unired} we have that the number of associative and quasitrivial $n$-ary operations that have exactly one neutral element (resp.\ that have no neutral element) on $[k]$ is exactly the number of associative and quasitrivial binary operation on $[k]$ that have a neutral element (resp.\ that have no neutral element). This number is given by $q_1^2(k)$ (resp.\ $q^2(k)-q_1^2(k)$). Also, if $n$ is even, then by Theorem \ref{thm:ack1}(a) and Proposition \ref{prop:arity}(a) we conclude that $q^n(k) = q^2(k)$ and $q_2^n(k) = 0$.

{{Now, suppose that $n$ is odd}}. By Corollary \ref{prop:countaqe} and Propositions \ref{prop:CDM} and \ref{prop:ae} we have that $q^n_{2}(k) = \frac{a_1^2(k)-q_1^2(k)}{2} = \binom{k}{2}q^2(k-2)$. Finally, by Proposition \ref{prop:arity} we have that $q^n(k) = q_{0}^n(k) + q^n_1(k) + q^n_{2}(k) = q^2(k) + \binom{k}{2}q^2(k-2)$.
\end{proof}

\begin{table}[htbp]\label{tab:1}
\[
\begin{array}{|c|rrrrrr|}
\hline k & q^2(k) & q_1^2(k) & q^n_{0}(k) & q^n_{2}(k) & q^n(k) & a_1^2(k)\\
\hline 1 & 1 & 1 & 0 & 0 & 1 & 1\\
2 & 4 & 2 & 2 & 1 & 5 & 4\\
3 & 20 & 12 & 8 & 3 & 23 & 18\\
4 & 138 & 80 & 58 & 24 & 162 & 128\\
5 & 1{\,}182 & 690 & 492 & 200 & 1{\,}382 & 1{\,}090\\
6 & 12{\,}166 & 7{\,}092 & 5{\,}074 & 2{\,}070 & 14{\,}236 & 11{\,}232\\
\hline
\mathrm{OEIS}^{\mathstrut} & \mathrm{A292932} & \mathrm{A292933} & \mathrm{A308352} & \mathrm{A308354} & \mathrm{A308362} & \mathrm{A308351}\\
\hline
\end{array}
\]
\caption{First few values of $q^2(k)$, $q_1^2(k)$, $ q^n_{0}(k)$, $q^n_{2}(k)$, $q^n(k)$ and $a_1^2(k)$}
\label{tab:q}
\end{table}

\section{Symmetric operations}

In this section we refine our previous results to the subclass of associative and quasitrivial operations that are symmetric, and present further enumeration results accordingly.

We first recall and establish  some auxiliary results.

\begin{fact}\label{lem:surj}
Suppose that $F\colon X^n \to X$ is associative and surjective.
If it is reducible to an associative operation $G\colon X^2 \to X$, then $G$ is surjective.
\end{fact}

\begin{lemma}[see {\cite[Lemma 3.6]{DevKiMar17}}]\label{lem:sym}
Suppose that $F\colon X^n \to X$ is associative, symmetric, and reducible to an associative and surjective operation $G\colon X^2 \to X$. Then $G$ is symmetric.
\end{lemma}

\begin{proposition}\label{prop:sym}
If $F\colon X^n \to X$ is associative, quasitrivial, and symmetric, then it is reducible to an associative, surjective, and symmetric operation $G\colon X^2 \to X$. Moreover, if $X=[k]$, then $F$ has a neutral element.
\end{proposition}

\begin{proof}
By Corollary \ref{cor:main1},  $F$ is reducible to an associative operation $G\colon X^2 \to X$. By Fact \ref{lem:surj} and Lemma \ref{lem:sym}, it follows that $G$ is surjective and symmetric.

For the moreover part, we only have two cases to consider.
\begin{itemize}
\item If $G$ is quasitrivial, then by \cite[Theorem 3.3]{CouDevMar2} it follows that $G$ has a neutral element, and thus $F$ also has a neutral element.
\item If $G$ is not quasitrivial, then  by Proposition \ref{prop:arity} and Theorem \ref{thm:unired} $F$ has in fact two neutral elements.\qedhere
\end{itemize}
\end{proof}

\begin{proposition}[see {\cite[Corollary 4.10]{Ack}}]\label{thm:ack}
An operation $F\colon X^n \to X$ is associative, quasitrivial, symmetric, and reducible to an associative and quasitrivial operation $G\colon X^2 \to X$ if and only if there exists a total ordering $\preceq$ on $X$ such that $F = max^n_{\preceq}$.
\end{proposition}

\begin{proposition}\label{thm:deg}
A quasitrivial operation $F\colon [k]^n\to [k]$ is associative, symmetric, and reducible to an associative and quasitrivial operation $G\colon [k]^2 \to [k]$ if and only if  $|F^{-1}| = (1,2^n-1,\ldots,k^n-(k-1)^n)$.
\end{proposition}

\begin{proof}
(Necessity) Since $G$ is quasitrivial, it is surjective and hence by Lemma \ref{lem:sym} it is symmetric. Thus, by Proposition \ref{thm:ack} there exists a total ordering $\preceq$ on $X$ such that $G(x,y) = \max^2_{\preceq}(x,y)$ for all $x,y \in [k]$. {{Hence $F=\max_{\preceq}^n$, which has an annihilator, and the proof of the necessity then follows by Proposition \ref{prop:an}.}}

(Sufficiency) We proceed by induction on $k$. The result clearly holds for $k=1$. Suppose that it holds for some $k\geq 1$ and let us show that it still holds for $k+1$. Assume that $F\colon [k+1]^n\to [k+1]$ is quasitrivial and that
\[
|F^{-1}|=(1,2^n-1,\ldots,(k+1)^n-k^n).
\]
Let $\preceq$ be the total ordering on $[k+1]$ defined by
\[
x\preceq y \ \text{ if and only if } \ |F^{-1}(x)|\leq|F^{-1}(y)|,
\]
and let $z=\max^{k+1}_{\preceq}(1,\ldots ,k+1)$. Clearly, $F'=F|_{([k+1]\setminus\{z\})^n}$ is quasitrivial and $|F'^{-1}|=(1,2^n-1,\ldots,k^n-(k-1)^n)$. By induction hypothesis we have that $F'=\max^n_{\preceq'}$, where $\preceq'$ is the restriction of $\preceq$ to $[k+1]\setminus\{z\}$. By Proposition \ref{prop:an}, $|F^{-1}[z]|=(k+1)^n-k^n$ and thus $F=\max^n _{\preceq}$.
\end{proof}

We can now state and prove the main result of this section.

\begin{theorem}\label{thm:sym}
Let $F\colon X^n\to X$ be an associative, quasitrivial, symmetric operation. The following assertions are equivalent.
\begin{enumerate}
  \item[(i)] $F$ is reducible to an associative and quasitrivial operation $G\colon X^2 \to X$.
  \item[(ii)] There exists a total ordering $\preceq$ on $X$ such that $F$ is $\preceq$-preserving.
  \item[(iii)] There exists a total ordering $\preceq$ on $X$ such that $F = max^n_{\preceq}$.
\end{enumerate}
Moreover, when $X = [k]$, each of the assertions $\textrm{(i)}-\textrm{(iii)}$ is equivalent to each of the following assertions.
\begin{enumerate}
  \item[(iv)] $F$ has exactly one neutral element.
  \item[(v)]  $|F^{-1}| = (1,2^n-1,\ldots,k^n-(k-1)^n)$.
\end{enumerate}
Furthermore, the total ordering $\preceq$ considered in assertions $\textrm{(ii)}$ and $\textrm{(iii)}$ is uniquely defined as follows:
\begin{equation}\label{eq:degsym}
x\preceq y \quad \mbox{if and only if} \quad |G^{-1}[x]|\leq |G^{-1}[y]|, \qquad x,y\in [k].
\end{equation}
Moreover, there are $k!$ operations satisfying any of the conditions $\textrm{(i)}-\textrm{(v)}$.
\end{theorem}

\begin{proof}
$\textrm{(i)} \Rightarrow \textrm{(iii)}$. This follows from Proposition \ref{thm:ack}.

$\textrm{(iii)} \Rightarrow \textrm{(ii)}$. Obvious.

$\textrm{(ii)} \Rightarrow \textrm{(i)}$. By Corollary \ref{cor:main1} we have that $F$ is reducible to an associative operation $G\colon X^2 \to X$. Suppose to the contrary that $G$ is not quasitrivial. From Theorem \ref{thm:unired} and Proposition \ref{prop:arity}, it then  follows that $F$ has two neutral elements $e_1,e_2\in X$, {{which contradicts Remark \ref{rem:ND}}}.

$\textrm{(i)} \Leftrightarrow \textrm{(v)}$. This follows from Proposition \ref{thm:deg}.

$\textrm{(i)} \Rightarrow \textrm{(iv)}$. This follows from Theorem \ref{thm:unired} and Proposition \ref{prop:sym}.

$\textrm{(iv)} \Rightarrow \textrm{(i)}$. This follows from Lemma \ref{lemma:Dud3} and Theorem \ref{thm:unired}.

The rest of the statement follows from \cite[Theorem 3.3]{CouDevMar2}.
\end{proof}

Now, let us illustrate Theorem \ref{thm:sym} for binary operations by means of their contour plots. In Figure~\ref{fig:3ab} (left),  we represent the contour plot of an operation $G\colon X^2\to X$ using the usual total ordering $\leq$ on $X=\{1,2,3,4\}$. In Figure~\ref{fig:3ab} (right) we represent the contour plot of $G$ using the total ordering $\preceq$ on $X$ defined by \eqref{eq:degsym}. We then observe that $G=\max_{\preceq}^2$, which shows by Theorem \ref{thm:sym} that $G$ is associative, quasitrivial, and symmetric.

\setlength{\unitlength}{3.5ex}
\begin{figure}[htbp]
\begin{center}
\begin{small}
\null\hspace{0.03\textwidth}
\begin{picture}(6,6)
\put(0.5,0.5){\vector(1,0){5}}\put(0.5,0.5){\vector(0,1){5}}
\multiput(1.5,0.45)(1,0){4}{\line(0,1){0.1}}%
\multiput(0.45,1.5)(0,1){4}{\line(1,0){0.1}}%
\put(1.5,0){\makebox(0,0){$1$}}\put(2.5,0){\makebox(0,0){$2$}}\put(3.5,0){\makebox(0,0){$3$}}
\put(4.5,0){\makebox(0,0){$4$}}
\put(0,1.5){\makebox(0,0){$1$}}\put(0,2.5){\makebox(0,0){$2$}}\put(0,3.5){\makebox(0,0){$3$}}
\put(0,4.5){\makebox(0,0){$4$}}
\multiput(1.5,1.5)(0,1){4}{\multiput(0,0)(1,0){4}{\circle*{0.2}}}
\drawline[1](1.5,4.5)(1.5,1.5)(4.5,1.5)\drawline[1](2.5,3.5)(2.5,2.5)(3.5,2.5)\drawline[1](2.5,4.5)(4.5,4.5)(4.5,2.5)
\end{picture}
\hspace{0.1\textwidth}
\begin{picture}(6,6)
\put(0.5,0.5){\vector(1,0){5}}\put(0.5,0.5){\vector(0,1){5}}
\multiput(1.5,0.45)(1,0){4}{\line(0,1){0.1}}%
\multiput(0.45,1.5)(0,1){4}{\line(1,0){0.1}}%
\put(1.5,0){\makebox(0,0){$3$}}\put(2.5,0){\makebox(0,0){$2$}}\put(3.5,0){\makebox(0,0){$4$}}
\put(4.5,0){\makebox(0,0){$1$}}
\put(0,1.5){\makebox(0,0){$3$}}\put(0,2.5){\makebox(0,0){$2$}}\put(0,3.5){\makebox(0,0){$4$}}
\put(0,4.5){\makebox(0,0){$1$}}
\multiput(1.5,1.5)(0,1){4}{\multiput(0,0)(1,0){4}{\circle*{0.2}}}
\drawline[1](1.5,2.5)(2.5,2.5)(2.5,1.5)\drawline[1](1.5,3.5)(3.5,3.5)(3.5,1.5)\drawline[1](1.5,4.5)(4.5,4.5)(4.5,1.5)
\end{picture}
\end{small}
\caption{An associative, quasitrivial, and symmetric binary operation $G$ on $X=\{1,2,3,4\}$ whose values on $(1,1), (2,2), (3,3)$ and $(4,4)$ are $1,2,3$ and $4$, respectively.}
\label{fig:3ab}
\end{center}
\end{figure}
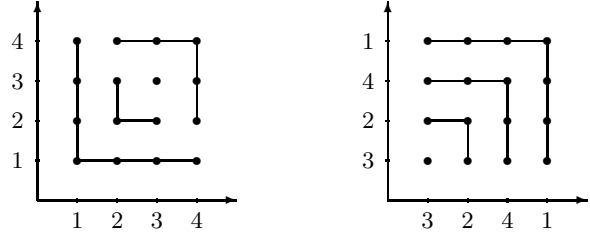

Based on this example, we illustrate a simple test to check whether an operation $F\colon [k]^n\to [k]$ is associative, quasitrivial, symmetric, and has exactly one neutral element. First,  construct the unique  weak ordering  $\precsim$ on $[k]$ from the preimage sequence $|F^{-1}|$, {{namely}},  $x\precsim y$ if $|F^{-1}[x]|\leq |F^{-1}[y]|$. Then, check if $\precsim$ is a total ordering and if $F$ is the maximum operation for $\precsim$.

We denote the class of associative, quasitrivial, symmetric operations $G\colon X^2 \to X$ that have a neutral element $e\in X$ by $QS_1^2(X)$. Also, we denote by $AS_1^2(X)$ the class of symmetric operations $G\colon X^2 \to X$ that belong to $A_1^2(X)$. It is not difficult to see that $QS_1^2(X)\subseteq AS_1^2(X)$. {{In fact, $G\in QS_1^2(X)$ if and only if $G\in AS_1^2(X)$ and $|{G}^{-1}[e]| = 1$, where $e$ is the neutral element for $G$.}}

For each integer $k\geq 2$, let $qs^n(k)$ denote the number of associative, quasitrivial, and symmetric $n$-ary operations on $[k]$. Also, denote by $as_1^2(k)$ the size of $AS_1^2([k])$. From Theorems \ref{thm:unired} and \ref{thm:sym} it follows that $qs^2(k) = |QS_1^2([k])| = k!$. Also, it is easy to check that $as_1^2(2)=4$. The remaining terms of the sequence are given in the following proposition.

\begin{proposition}\label{prop:ase}
For every integer $k\geq 3$, $as_1^2(k) = qs^2(k) + kqs^2(k-1) = 2k!$.
\end{proposition}

\begin{proof}
As observed $QS_1^2([k])\subseteq AS_1^2([k])$. So let us enumerate the operations in $ AS_1^2([k])\setminus QS_1^2([k])$. There are $k$ ways to choose the element $x\in [k]$ such that $G(x,x) = e$ and $G(x,y) = G(y,x) = y$ for all $y\in [k]\setminus \{x,e\}$. Moreover, the restriction of $G$ to $([k]\setminus\{x\})^2$ belongs to $QS_1^2([k]\setminus \{x\})$, and we have $qs^2(k-1)$ possible such restrictions. Thus $as_1^2(k)=qs^2(k) + kqs^2(k-1)$. By Theorems \ref{thm:unired} and \ref{thm:sym} it then follows that $as_1^2(k) = k! + k(k-1)! = 2k!$.
\end{proof}

For any integer $k\geq 2$ let $qs^n_1(k)$ denote the number of associative, quasitrivial, and symmetric $n$-ary operations that have exactly one neutral element on $[k]$. Also, let $qs^n_{2}(k)$ denote the number of associative, quasitrivial, and symmetric $n$-ary operations that have two neutral elements on $[k]$.

\begin{proposition}
For each integer $k\geq 2$, $qs^n_1(k) = qs^2(k) = k!$. Moreover,
$
qs^n_{2}(k) = \frac{k!}{2},
$
and
$
qs^n(k) = \frac{3k!}{2}.
$
\end{proposition}

\begin{proof}
By Theorems \ref{thm:sym} and \ref{thm:unired} and Lemma \ref{lem:sym} we have that the number of associative, quasitrivial, and symmetric $n$-ary operations that have exactly one neutral element on $[k]$ is exactly the number of associative, quasitrivial, and symmetric binary operations on $[k]$. By Theorems \ref{thm:unired} and \ref{thm:sym} this number is given by $qs^2(k) = k!$. Also, by Corollary \ref{prop:countaqe}, Proposition \ref{prop:ase}, and Theorems \ref{thm:unired} and \ref{thm:sym}, we have that $q^n_{2}(k) = \frac{as_1^2(k)-qs^2(k)}{2} = \frac{k!}{2}$ and by Proposition \ref{prop:arity} we have that $q^n(k) = qs^n_1(k) + qs^n_{2}(k) = \frac{3k!}{2}$.
\end{proof}

\begin{remark}
Recall that an operation $F\colon X^n \to X$ is said to be \emph{bisymmetric} if
\[
F(F(\mathbf{r}_1),\ldots,F(\mathbf{r}_n)) ~=~ F(F(\mathbf{c}_1),\ldots,F(\mathbf{c}_n))
\]
for all $n\times n$ matrices $[\mathbf{c}_1 ~\cdots ~\mathbf{c}_n]=[\mathbf{r}_1 ~\cdots ~\mathbf{r}_n]^T\in X^{n\times n}$. In \cite[Corollary 4.9]{DevKiMar17} it was shown that associativity and bisymmetry are equivalent for operations $F\colon X^n \to X$ that are quasitrivial and symmetric. Thus, we can replace associativity with bisymmetry in Theorem \ref{thm:sym}.
\end{remark}

\section{Conclusion}

In this paper we proved that any quasitrivial $n$-ary semigroup is reducible to a semigroup. Furthermore, we showed that a quasitrivial $n$-ary semigroup is reducible to a unique quasitrivial semigroup if and only if it has at most one neutral element. Moreover, we characterized the class of quasitrivial (and symmetric) $n$-ary semigroups that have at most one neutral element. Finally, when the underlying set is finite, this work led to four new integer sequences in the Sloane's OEIS \cite{Slo}, namely, A308351, A308352,  A308354, and A308362.

Note however that there exist idempotent $n$-ary semigroups that are not reducible to a semigroup (for instance, consider the idempotent associative operation $F\colon \mathbb{R}^3 \to \mathbb{R}$ defined by $F(x,y,z) = x-y+z$ for all $x,y,z\in \mathbb{R}$). This naturally asks for necessary and sufficient conditions under which an idempotent $n$-ary semigroup is reducible to a semigroup. This and other related questions constitute topics for future research.

\section*{Acknowledgements}

Both authors would like to thank Jean-Luc Marichal and the anonymous referee for their  useful comments and insightful remarks that helped improving the current paper.

\end{document}